\documentclass[a4paper,10pt]{amsart}
\usepackage[utf8]{inputenc}
\usepackage{ulem} %\fk
\usepackage{verbatim} %\fk
\usepackage{latexsym}
\usepackage{amssymb,amsmath,amsopn}
\usepackage{bm}
\usepackage{url}
\usepackage[dvips]{graphicx}   %To insert ps figures
\usepackage{color,epsfig}      %To insert ps figures

\newtheorem{theorem}{Theorem}

\newtheorem{proposition}{Proposition}
\theoremstyle{definition}
\newtheorem{definition}{Definition}
\newtheorem{remark}{Remark}

\newtheorem{corollary}{Corollary}

\definecolor{green}{rgb}{0.0, 0.75, 0.25} %\fk
\newcommand{\fk}{\color{green}} %\fk
\def\eea{\end{eqnarray}}
\renewcommand{\emph}{\textit}

\DeclareMathOperator{\inter}{int}
\DeclareMathOperator{\bd}{bd}

\title[A mono-monostatic polyhedron with point masses at the 21 vertices]{A mono-monostatic polyhedron with point masses at the 21 vertices}

\author{Gábor Domokos} \email{domokos@iit.bme.hu}
\address{MTA-BME Morphodynamics Research Group / BME Department of Mechanics and Structures, Budapest University of Technology and Economics, Műegyetem rkp. 3., Budapest, Hungary, 1111}

\author{Flórián Kovács} \email{kovacs.florian@emk.bme.hu}
\address{MTA-BME Morphodynamics Research Group / BME Department of Structural Mechanics, Budapest University of Technology and Economics, Műegyetem rkp. 3., Budapest, Hungary, 1111}

\thanks{Support of the NKFIH Hungarian Research Fund Grant No. 134199 and Grant No. TKP2021 BME-NVA is kindly acknowledged. }


\date{}

\begin{document}

\begin{picture}(0,0)(0,0)
 \put(20,20){{\fk (Accepted for publication in the American Mathematical Monthly)}}
\end{picture}

\title{Conway's spiral and a discrete Gömböc with 21 point masses}
%\markright{A discrete 21-vertex Gömböc}
%\author{Gábor Domokos and Flórián Kovács}

\begin{abstract}
We show an explicit construction in 3 dimensions for a convex, mono-monostatic polyhedron (i.e., having exactly one stable and one unstable equilibrium) with 21 vertices and 21 faces. This polyhedron is a 0-skeleton, with equal masses located at each vertex. The above construction serves as an upper bound for the minimal number of faces and vertices of mono-monostatic 0-skeletons and complements the recently provided lower bound of 8 vertices. This is the first known construction of a mono-monostatic polyhedral solid. We also show that a similar construction for homogeneous distribution of mass cannot result in a mono-monostatic solid.
\end{abstract}

\maketitle

\section{Introduction.}\label{sec:intro}

Dice have been used since millennia to generate random integers \cite{dice}. The most common geometric form of a dice is a convex polyhedron.
Throwing dice is a mechanical experiment executed on a horizontal plane, and in the experiment we select randomly from among the \emph{stable equilibrium points} lying on the faces of the polyhedron. Dice are called \emph{fair} if the probabilities to rest on any face (after a random throw) are equal  \cite{Diaconis}, otherwise they are called \emph{loaded} \cite{DawsonFinbow}. Despite being associated with mechanical experiments with solids,
the concept of static equilibrium may also be defined on polyhedra in purely geometric terms \cite{balancing}:
\begin{definition}\label{def1}
Let $P$ be a $3$-dimensional convex polyhedron, let $\inter P$ and $\bd P$ denote its interior and boundary, respectively and let $o \in \inter P$.  Let $P$ be associated with some mass distribution $\mu(P)$. Then we say that the pair $(P,o)$ is a \emph{polyhedral solid} if $o$ coincides with the center of mass $c$ of $\mu(P)$. We call $q \in \bd P$ an \emph{equilibrium point} of $P$ with respect to $o$ (or, alternatively, an equilibrium point of the polyhedral solid $(P,o)$) if the plane $H$ through $q$ and perpendicular to $[o,q]$ supports $P$ at $q$. In this case $q$ is \emph{nondegenerate} if $H \cap P$ is the (unique)  vertex, edge, or face of $P$, respectively, that contains $q$ in its relative interior. A nondegenerate equilibrium point $q$ is called \emph{stable, saddle-type} or \emph{unstable}, if $\dim (H \cap P) = 2,1$ or $0$, respectively.
\end{definition}

\begin{remark}
Throughout this paper we will briefly refer to polyhedral solids as polyhedra.
\end{remark}
\begin{remark}
The definition for static equilibria of convex polygons and polygonal solids is analogous; however, in that case we only distinguish between (generic) stable equilibrium points in the interior of the edges and unstable equilibrium points at the vertices.
\end{remark}
Throughout this paper we deal only with nondegenerate equilibrium points with respect to the center of mass of polyhedral solids; thus, we have $o=c$, in which case equilibrium points gain intuitive interpretation as locations on $\bd P$
where $P$ may be balanced if it is supported on a horizontal surface (identical to the \emph{support plane} mentioned in Definition \ref{def1}) without friction in the presence of uniform gravity.  
\begin{definition}\label{def:skeleton}
Polyhedral solids with some special material distributions are defined as \emph{polyhedral $h$-skeletons} as follows: 0-skeletons have mass uniformly distributed on their vertices,  1-skeletons have mass uniformly distributed on the edges, 2-skeletons have mass uniformly distributed  on the faces, whereas 3-skeletons have uniform density. 3-skeletons are also referred to as \emph{homogeneous polyhedra}. In two dimensions, we use the term \emph{polygonal $h$-skeleton} with $h=0,1,2$ only and 2-skeletons are also referred to as \emph{homogeneous polygons}.
\end{definition}

As discussed in Definition \ref{def1}, equilibrium points may belong to three stability types: faces may carry stable equilibria, vertices may carry unstable equilibria, and edges may carry saddle-type equilibria. Denoting their respective numbers by $S,U,H$, by the Poincar\'e-Hopf formula \cite{Milnor, Langi} for a convex  polyhedral solid one obtains the following relation:
\begin{equation}\label{Poincare}
S+U-H=2,
\end{equation}
implying that any two of these numbers determine the third. We will refer to $(S,U)^E$ as the (primary) \emph{equilibrium class} of the polyhedron $P$
\cite{VarkonyiDomokos, balancing}. Analogously, we denote respective numbers of faces, vertices, and edges of $P$ by $F,V,E$ and for these the Euler
formula 
\begin{equation}\label{Euler}
F+V-E=2,
\end{equation}
holds, defining the combinatorial class $(F,V)^C$ of the polyhedron. In two dimensions, for convex polygons $S=U$ and $F=V$ always hold, hence equilibrium classes may be defined by the number of unstable equilibria as $(U)^E$ and combinatorial classes by the number of vertices as $(V)^C$.

\subsection{Results on homogeneous monostatic polyhedra.}

Polyhedral solids in equilibrium classes $(1,U)^E$ and in classes $(S,1)^E$ are collectively called \emph{monostatic} \cite{balancing}.
\subsubsection{Monostable homogeneous polyhedra}
\begin{definition}\label{monostable}
Polyhedral solids in equilibrium classes $(1,U)^E$ are called \emph{monostable} \cite{balancing}.
We denote the smallest number of faces among all monostable, convex, homogeneous  polyhedra by $F^S$ and we denote the smallest number of vertices among all monostable, convex, homogeneous  polyhedra by $V^S$.
\end{definition}

While monostable, homogeneous polyhedra have attracted considerable mathematical interest, $F^S$ and $V^S$ are not known. On the other hand, some bounds do exist.
In 1967, Conway and Guy \cite{Conway}
offered the first upper bound for $F^S$ and $V^S$ by describing such an object with $F=19$ faces and $V=34$ vertices to which we henceforth refer as the \emph{Conway-Guy polyhedron}. The face and vertex numbers associated with  the Conway-Guy polyhedron were improved by Bezdek \cite{Bezdek} to $(18,18)^C$ and later by Reshetov \cite{Reshetov} to $(14,24)^C$. These values of $F$ and $V$ determine the best known \emph{upper bounds} for a homogeneous monostable polyhedron, so we have  $F^S \leq 14 , V^S \leq 18$. Even less is known about the lower bounds: the only known result is due to Conway \cite{Dawson} who proved that each homogeneous tetrahedron has at least two stable equilibria, from which $F^S, V^S \geq 5$ follows.

\subsubsection{Mono-unstable  homogeneous polyhedra}

\begin{definition}\label{monounstable}
Polyhedral solids in equilibrium classes $(S,1)^E$ are called \emph{mono-unstable} \cite{balancing}.
We denote the smallest number of faces among all mono-unstable, convex, homogeneous  polyhedra by $F^U$ and we denote the smallest number of vertices among all mono-unstable, convex, homogeneous  polyhedra by $V^U$.
\end{definition}

The Conway-Guy polyhedron has, beyond the single stable position on one face, $U=4$ unstable equilibria at the $4$ vertices of the very same face. The first example for a mono-unstable polyhedron  was demonstrated in \cite{balancing}, having $F=18$ faces and $V=18$ vertices and in the same paper it was proven that a homogeneous tetrahedron cannot be mono-unstable. Thus, for the minimal numbers $F^U, V^U$ for the faces and vertices that a homogeneous, mono-unstable polyhedron may have, the following bounds apply:
 $ 5 \leq F^U \leq 18$, $5 \leq V^U \leq  18$.
 
\subsubsection{Mono-monostatic homogeneous polyhedra}
\begin{definition}\label{monomono}
Polyhedral solids in equilibrium class $(1,1)^E$ are called \emph{mono-monostatic} \cite{balancing,VarkonyiDomokos}.
We denote the smallest number of faces among all mono-monostatic, convex, homogeneous  polyhedra by $F^{\star}$ and we denote the smallest number of vertices among all mono-monostatic, convex, homogeneous  polyhedra by $V^{\star}$.
\end{definition}

While the existence of homogeneous, mono-monostatic polyhedra has been proven \cite{Langi}, no example is known. The only known convex, homogeneous, mono-monostatic objects are non-polyhedral,  the first example is called G\"{o}mb\"{o}c \cite{VarkonyiDomokos}. This implies that for the minimal numbers $F^{\star}, V^{\star}$ for the faces and vertices of a homogeneous mono-monostatic polyhedron the only known bounds are $F^{\star}, V^{\star} \geq 5$.

\subsection{0-skeletons and the main result.}
Here we highlight a new aspect of this problem: instead of looking at the homogeneous case with uniform mass distribution, we consider polyhedral 0-skeletons with unit masses at the vertices as introduced in \cite{Dawson3}.
\begin{remark} \label{rem_h0_h}
Definitions \ref{monostable}, \ref{monounstable} and \ref{monomono} of the monostable, mono-unstable and mono-monostatic properties apply to all polyhedra, regardless of their mass distribution. In particular, it applies both to homogeneous polyhedra and to polyhedral 0-skeletons.
\end{remark}
Here again we seek the minimal face and vertex numbers to obtain monostatic polyhedra and the corresponding numbers are defined below.
 \begin{definition}\label{0skeleton}
We denote the smallest number of faces among all monostable (mono-unstable, mono-monostatic), convex polyhedral 0-skeletons by $F_0^S$ ($F_0^U, F_0^{\star})$, respectively.  We denote the smallest number of vertices  among all monostable (mono-unstable, mono-monostatic), convex polyhedral 0-skeletons by $V_0^S$ ($V_0^U, V_0^{\star})$, respectively.
\end{definition} 
The problem of finding these minima may appear, at first sight, almost `unsportingly' easy as compared with the homogeneous case. However, the minimal vertex number $V^{\star}_0$ and face number $F^{\star}_0$  to produce  a mono-monostatic polyhedral 0-skeleton are not known. Even more curiously, the minimal number of vertices for a mono-monostatic,  \emph{polygonal} 0-skeleton (in 2 dimensions) is not known either.

The first related results have been reported in \cite{Bozoki} where $V^{U}_0 \geq 8$ was proven (implying, via the theorem of Steinitz \cite{Steinitz1}, the lower bound $F^{U}_0 \geq 6$). This result also implies the bounds $F^{\star}_0 \geq 6, V^{\star}_0 \geq 8$
for mono-monostatic polyhedral 0-skeletons.

In this paper we explain the background and show some constructions which may inspire further research. In particular, by providing an explicit construction of a mono-monostatic polyhedral 0-skeleton with 21 faces and 21 vertices, we prove

\begin{theorem}\label{th1}
$F^{\star}_0,  V^{\star}_0 \leq 21.$
\end{theorem}
Our example, illustrated in Figure \ref{fig:intro}(c) and defined on line 3 of Table~\ref{tb:monomono},  appears to be the first discrete construction of a mono-monostatic object and it may help to inspire thinking about
the bounds  $F^{\star}, V^{\star}$ for the homogeneous case.

The paper is structured as follows. In Section~\ref{spirals} we explain the geometric idea behind Conway's classical construction and how
this idea may be generalized in various directions. In Section~\ref{skeletons}, by relying on an idea by Dawson \cite{Dawson}, we describe the construction for a mono-monostatic 0-skeleton in two dimensions, having $V_0=11$ vertices and then we proceed to prove Theorem \ref{th1} by providing the construction of the mono-monostatic 0-skeleton in three dimensions. In Section~\ref{other} we show the connection to other problems, including the mechanical complexity of polyhedra, and also point out why the particular geometry of our constructions may not be applied to the construction of a homogeneous mono-monostatic polyhedron. In Section~\ref{sum} we draw conclusions.

\section{The geometry of Conway spirals.}\label{spirals}
\begin{figure}[ht]
\begin{center}
\includegraphics[width=\textwidth]{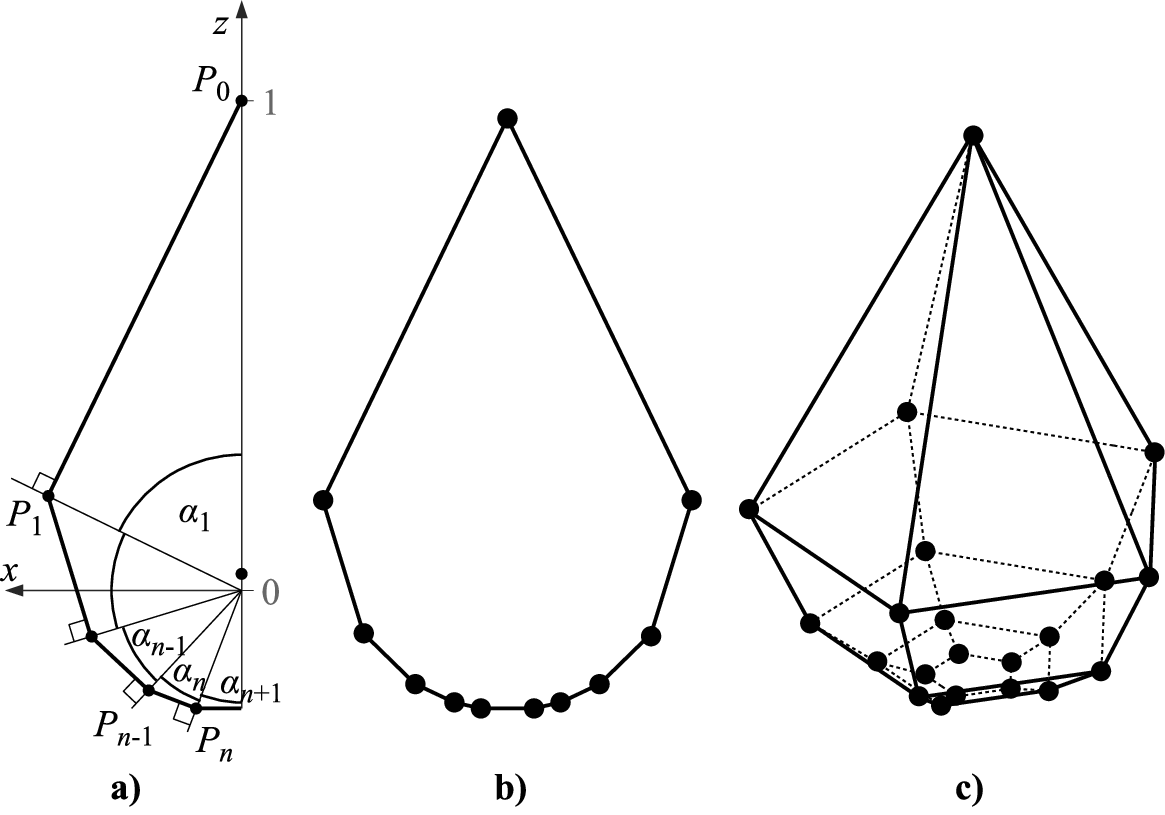}
\caption{Construction of symmetric, mono-monostatic polygons and polyhedra; a) geometry of the Conway spiral $P_0 , \ldots, P_n$.  $P_0$ is fixed at $z=1$ and each radius $OP_i$ is perpendicular to the corresponding edge $P_{i-1}P_i$. The geometry of the spiral is uniquely described in terms of $n$ angular variables $\alpha_1, \ldots, \alpha_n$: $\alpha_i~=~\mbox{\textit{constant}}$ results in a classical Conway spiral; b) 2D mirror-symmetric mono-monostatic polygon with 11 vertices for $n=5$ and $k=2$, see Table~\ref{tb:monomono}, line 6 for numerical data; c) 3D mono-monostatic polyhedron with 5-fold rotational symmetry for $n=4$ and $k=5$, see Table~\ref{tb:monomono}, line 3 for numerical data.}
\label{fig:intro}
\end{center}
\end{figure}

\subsection{The classical Conway double spiral and the Conway-Guy monostable polyhedron.}
The essence of the Conway-Guy polyhedron is a remarkable planar construction to which we will briefly refer as the \emph{Conway spiral}, illustrated in Figure~\ref{fig:intro}(a) and which we define below.
\begin{definition}\label{conway_general}
A Conway $n$-spiral is an open polygonal curve with $n+1$ vertices  $P_0 , \ldots, P_n$ with $\angle O P_i P_{i-1} = \pi/2$, $i = 1\ldots n$,
where the point $O$ is  the origin of the coordinate system, all points $P_i$ with coordinates $(x_i, y_i, z_i)$ lie in the plane $[xz]$, the coordinates of $P_0$ are $(x_0,y_0,z_0)=(0,0,1)$
and all remaining vertices $P_i$, $i=1,2, \dots n$ have positive $x$ coordinates. The central angles  $\alpha_i$ are defined for $i=1,2, \dots n$ as $\alpha _i= \angle P_i O P_{i-1}$ and the central angle $\alpha_{n+1}$ is given as $\alpha_{n+1}=\pi - \sum_{i=1}^n\alpha_i$. We say that the edge $P_iP_{i-1}$ $(i=1,2, \dots n)$ belongs to the \emph{upper part} of the Conway $n$-spiral if and only if $x_i \geq x_{i-1}$.
\end{definition}
This leads to the construction of a special class of convex polygons which we define below.
\begin{definition}
A double Conway $n$-spiral is a convex polygon with $2n+1$ vertices which is obtained by reflecting a Conway spiral
to the $[yz]$ plane. The coordinates for the center of mass $c$ are denoted by $[x_c,y_c,z_c]$.
\end{definition}

Due to reflection symmetry, for double Conway spirals both with homogeneous material distribution and also for 0-skeletons, we have $x_c=y_c=0$. We also note that, due to the special design, the double Conway spiral (both for homogeneous material distribution and for 0-skeletons) is monostatic in the $[xz]$ plane if and only if $z_c<0$. 
The most straightforward construction of a Conway spiral uses uniform central angles:
\begin{definition}
We call a Conway $n$-spiral \emph{classical} if all central angles $\alpha _i= \angle P_i O P_{i-1}$, $(i=1,2, \dots n+1)$ are equal, i.e., we have
\begin{equation}\label{eq:conway}
\alpha _1 = \alpha _2 =\dots =\alpha _{n+1}.
\end{equation}
\end{definition}
In a classical Conway spiral all triangles $P_iP_{i-1}O$ are similar and
the original Conway-Guy polyhedron \cite{Conway} relies on a classical double  Conway spiral.
Classical Conway spirals form  a discrete family of open polygons, parametrized by the integer $n$, and  this also holds for classical double Conway spirals. None of these polygons associated with homogeneous material distribution is monostatic in the $[xz]$ plane, i.e., we have $z_c>0$ for all values of $n$,
since convex monostatic, homogeneous polygons do not exist \cite{DomokosRuina}. Still, the classical double Conway spiral may be regarded as a \emph{best shot} at a homogeneous, monostatic polygon with reflection symmetry. The same intuition suggests that a classical double Conway spiral may need minimal added `bottom weight' to become monostatic.

Conway and Guy added this bottom weight by extending the shape into 3D as an oblique prism and they computed the minimal value of $n$ necessary to make
this homogeneous oblique prism (with the cross-section of a classical Conway spiral) monostable as $n=8$, resulting in a homogeneous, convex polyhedron with 34 vertices and 19 faces.

\subsection{The exponential Conway double spiral and Dawson's monostable simplices in higher dimensions.}
The idea of the Conway spiral may be generalized to bear more fruit. In \cite{Dawson} Dawson, seeking monostatic simplices in higher dimensions, considered the following version:
\begin{definition}
We call a Conway spiral \emph{exponential} if the central angles $\alpha _i= \angle P_i O P_{i-1}$, $(i=1,2, \dots n+1)$ are given by
\begin{equation}\label{eq:dawson}
%\alpha_1 = c^{k-1}\alpha _k, \quad k=1,2, \dots n
\alpha_i = b^{1-i}\alpha _1, \quad i=1,2, \dots n \qquad\mbox{and}\qquad \alpha_{n+1}=\alpha_n
\end{equation}
and we refer to this open polygon as an exponential Conway $n$-spiral with parameter $b$.
\end{definition}
Dawson considered the double exponential Conway spiral with $2n+1$ vertices embedded into a  $2n$-dimensional space,  with the mirror images of the vertex $P_i$ defined as $P_{-i}$. He regarded the vertex vectors $\mathbf{x}_i=OP_i$, $i=-n, -n+1 \dots n$ of the exponential double Conway spiral as the \emph{face vectors} of a $2n$-dimensional simplex. (Face vectors $\mathbf{x}_i$ $(i=1,2, \dots F)$ may be associated with any polyhedron having $F$ faces $f_i$, $(i=1,2,\dots F)$ in the following manner: $\mathbf{x}_i$ is orthogonal to the face $f_i$ and has magnitude proportional to the area of $f_i$). Since in this case these vectors are coplanar, the simplex is degenerate (infinite). However, Dawson also added a small, generic $2n$-dimensional perturbation to the coordinates of the vertices of the spiral, to obtain the set of face vectors for a generic $2n$-dimensional simplex. To qualify as face vectors, any set of vectors must be balanced \cite{Minkowski}, i.e., we must have 
\begin{equation}\label{eq:sum}
    \sum_{i=-n}^{n}\mathbf{x}_i=0.
\end{equation}
 Dawson  proved \cite{Dawson} that for a homogeneous simplex (or a simplex interpreted as a 0-skeleton) supported on a horizontal plane, in the presence of gravity, the condition for tipping from face $f_i$ to $f_j$ can be written as
 \begin{equation}\label{Dawson1}
\lvert\mathbf{x}_i\rvert < \lvert\mathbf{x}_j\rvert\cos\theta_{ij},
\end{equation}
where $\theta_{ij}$ is the angle between $\mathbf{x}_i$ and $\mathbf{x}_j$. By using this 
\emph{tipping condition} he found that for $n=5, b=1.5$ the exponential Conway spiral (\ref{eq:dawson})  yields a set of balanced vectors
in the $[xz]$ plane, the small
generic, 10-dimensional (`out of plane') perturbation of which defines a 10-dimensional, homogeneous monostable simplex.

\section{Mono-monostatic 0-skeletons.}\label{skeletons}
\subsection{Double Conway spirals and planar 0-skeletons.}\label{ss:11gon}
Instead of considering double Conway spirals as homogeneous polygons, we can regard them as planar \emph{polygonal 0-skeletons}. Since there are relatively many vertices with negative $z$ coordinate and relatively few ones with positive $z$ coordinate, this interpretation appears to be a convenient manner to
add `bottom weight' to the double Conway spiral. In this interpretation as planar 0-skeletons, one may ask whether mono-monostatic double Conway spirals exist and, if yes, what is the minimal number of their vertices necessary to have this property. Dawson's result leads to the following

\begin{proposition}\label{prop1}
Mono-monostatic polygonal 0-skeletons with $V=11$ vertices exist. 
\end{proposition}
\begin{proof}
We will prove the proposition by showing that a 0-skeleton generated by the exponential double Conway 5-spiral with parameter $b=1.5$ is mono-monostatic. 
Since static balance equations for such a 0-skeleton coincide with (\ref{eq:sum}) and, based on the results presented in \cite{Bozoki}, the tipping condition (\ref{Dawson1}) is equivalent to prohibit an unstable equilibrium at vertex $v_i$, we can see that Dawson's geometric construction, the 0-skeleton generated by the  exponential double Conway 5-spiral with parameter 1.5, has $z_c<0$ and thus it defines a polygon  with $V=11$ vertices which is mono-monostatic. 
\end{proof}

One can ask whether this construction is optimal in two ways: whether there exists a smaller value of $n$ which defines a mono-monostatic 0-skeleton generated by the exponential double Conway spiral
and whether by keeping $n=5$, one may pick other values for $\alpha_i$ which
yield a center of mass with larger negative coordinate. The first question was answered in \cite{Dawson3} in the negative by proving that monostable simplices in $d<9$  dimensions do not exist. This implies that for $n<5$ no mono-monostatic 0-skeleton generated by a Conway spiral  exists, but nothing is known about the existence of mono-monostatic 10-gonal polygons as 0-skeletons since they cannot be represented by a symmetric double Conway spiral. To answer the second question we will use Definition \ref{conway_general} and consider general Conway spirals with \emph{arbitrary} $\alpha_i$
(satisfying the condition given in Definition \ref{conway_general}) and optimize this construction to seek the minimum of $z_c$. 
To verify the monostatic property of a given double Conway $n$-spiral, the coordinate $z_c$ of the center of mass  needs to be computed. In terms of coordinates $z_i$, we have from Figure~\ref{fig:intro}(a):
\begin{equation}
\label{eq:z_C}
 z_c = \dfrac{1+k \displaystyle\sum_{i=1}^{n} z_i}{1+kn},
\end{equation}
where $k=2$.
This can be rewritten in terms of variables $\boldsymbol{\alpha} = (\alpha_1 \ldots \alpha_n)$ as follows:
\begin{equation}
\label{eq:z_prodsum}
z_C(\boldsymbol{\alpha}) = \dfrac{1+k \displaystyle\sum_{i=1}^{n} \prod_{j=1}^i \cos\alpha_j\cdot\cos\left(\sum_{j=1}^i \alpha_j\right)}{1+kn},
\end{equation}

\begin{remark} \label{rem_k}
In the next section we will show constructions where formula (\ref{eq:z_C}) will be interpreted for higher values of $k$.
\end{remark}
We performed an optimization for $\boldsymbol{\alpha}$ in the following manner: The function $z_c(\boldsymbol{\alpha})$ can have a minimum if $\frac{\partial z_c(\boldsymbol{\alpha})}{\partial\alpha_i}=0$ for $i=1, \ldots n$.
One can express $\tan(\alpha_n)$ from $\frac{\partial z_c(\boldsymbol{\alpha})}{\partial\alpha_n}=0$ and, recursively, $\alpha_i = \alpha_i (\tan(\alpha_{i+1}), \ldots, \tan(\alpha_n))$ from $\frac{\partial z_c(\boldsymbol{\alpha})}{\partial\alpha_i}=0$. This yields a univariate polynomial equation for $\tan(\alpha_n)$ that can be solved numerically.

Using this algorithm, we found the shape in Figure~\ref{fig:intro}(b) (see Table~\ref{tb:monomono}, line 6 for computed values of $\boldsymbol{\alpha}$). Note that this result is an alternative proof for the existence of monostable $10$-dimensional simplices given by Dawson \cite{Dawson}. 

We remark that a similar optimization process of the Conway spiral is discussed in \cite{Minich} for the homogeneous case.

\subsection{Proof of Theorem \ref{th1}:  Conway $(n,k)$-spirals and mono-monostatic 0-skeletons in 3 dimensions.}\label{ss:21hedron}

\begin{proof}
In the first step of the proof we further generalize the concept of Conway spirals by considering out-of-plane, 3D arrangements,
defining convex polyhedra:

\begin{definition}
Let us consider a Conway $n$-spiral and rotate it around the $z$-axis $k>2$ times by the angle  $\beta=2\pi/k$. Since the vertex $P_0$
is lying on the $z$ axis, this operation generates  $V=kn+1$ vertices. Beyond the edges defined by the  $k$  Conway spirals, we add
$n$ regular $k$-gons in planes parallel to the $[xy]$ plane. This defines a convex polyhedron with $V=kn+1$ vertices,
$E=2nk$ edges, and $F=kn+1$ faces. We call such a polyhedron a Conway $(n,k)$-spiral and briefly denote it by $P_{n,k}$.
\end{definition}

\begin{remark}
Conway double $n$-spirals could be regarded as Conway $(n,2)$-spirals; however, they do not define polyhedra so we
will only use the $(n,k)$-notation for the case where $k>2$.
\end{remark}

If for $k=2$ the Conway double spiral defines  a mono-monostatic polygonal 0-skeleton then we expect that for higher values of $k$ we will obtain mono-monostatic polyhedral 0-skeletons.

The procedure of finding mono-monostatic 0-skeletons generated by Conway $(n,k)$-spirals is also based on (\ref{eq:z_prodsum}) and a similar optimization as introduced for $k=2$. Without expanding the procedure in detail, it should only be noted that mono-monostatic behaviour requires here that equilibrium points must be excluded not only on any of the edges of a Conway spiral but also on any of the faces between two adjacent spirals; the first and last condition is stronger in- and outside the upper part of the $(n,k)$-spiral, respectively.

We performed calculations in search of minimum $z_c$ that lead to different constructions; one of these constructions with $n=4, k=5$, i.e., $P_{4,5}$ is illustrated in Figure~\ref{fig:intro}(c).

Table~\ref{tb:monomono} summarizes the possible mono-monostatic objects with minimum required $k$ found by the above method ($v = kn+1$ stands both for the number of vertices or/and faces).
\begin{table}[!ht]
\begin{center}
\begin{tabular}{|c|ccclp{7.7cm}|}
 \hline
 no. & $n$ & $k$ & $v$ & $z_C$ & $(\alpha_{n+1},\alpha_n, \ldots, \alpha_1)$
 \\ \hline \hline
 1 & 2 & 25 & 51 & -0.00051277 & $(49.799,   49.799,   80.402)^\circ$ 
 \\ \hline
 2 & 3 & 8 & 25 & -0.0061413 & $(30.273, 30.273, 46.543, 72.912)^\circ$ 
 \\ \hline
 3 & 4 & 5 & 21 & -0.015354 & $(19.716, 19.716, 29.875, 44.519, 66.173)^\circ$ 
 \\ \hline
 4 & 5 & 4 & 21 & -0.029972 & $(13.494, 13.494, 20.336, 29.781, 43.215, 59.680)^\circ$ 
 \\ \hline
 5 & 7 & 3 & 22 & -0.042695 & $( 7.1815, 7.1815, 10.7864, 15.6392, 22.1409,$ $30.9129, 43.0793, 43.0788)^\circ$ 
 \\ \hline
 6 & 5 & 2$^\ast$ & 11 & -0.017984 & $( 13.201, 13.201, 19.890, 29.110, 42.172, 62.427)^\circ$ 
 \\ \hline
\end{tabular}
\vskip 5mm 
\caption{List of some mono-monostatic 0-skeletons $P_{n,k}$ with $D_k$ symmetry and $v = nk+1$ vertices; $z_c$ can be verified via (\ref{eq:z_prodsum}). $k=2$ marked by `$\ast$' is the two-dimensional case obtained from (\ref{eq:z_C}) and (\ref{eq:z_prodsum}) via numerical optimization. The minimum number of vertices for monostatic 3D rotational polyhedra is 21.}
\label{tb:monomono}
\end{center}
\end{table}
\end{proof}

We believe that this construction  is close to a (local) optimum, i.e., we think that this may be the mono-monostatic 0-skeleton defined by Conway $(n,k)$-spirals which has the least number of vertices. This, however, does not exclude the existence of mono-monostatic polyhedral 0-skeletons with smaller number of vertices which have less symmetry. Our construction provides 21 as an \emph{upper bound} for the minimal number of vertices and faces of a mono-monostatic polyhedral 0-skeleton. The lower bound for the number of vertices was given in \cite{Bozoki} as 8, from which a lower bound of 6 for the number of faces follows \cite{Steinitz1}.

\section{Connection to related other problems.}\label{other}
\subsection{Mechanical complexity of polyhedra.}
It is apparent that constructing monostatic polyhedra is not easy. In \cite{balancing} this general observation was formalized by introducing the \emph{mechanical complexity} $C(P)$ of a polyhedron $P$ as
\begin{equation}\label{complexity}
 C(P) = 2(V(P) + F(P) - S(P) - U(P)),
\end{equation}
where $V(P),F(P),S(P),U(P)$ stand for the number of vertices, faces, stable and unstable equilibrium points of $P$, respectively.
As described after (\ref{Poincare}), the equilibrium class of polyhedral solids with given numbers $S,U$ of stable and unstable equilibria is denoted by $(S,U)^E$ and the complexity of such class was defined in \cite{balancing} as
\begin{equation}\label{complexity1}
C(S,U)=\min \{C(P): P \in (S,U)^E\};
\end{equation}
however, the only material distribution considered in \cite{balancing} was uniform density. Here we generalize this concept for $h$-skeletons (see Definition \ref{def:skeleton}). To distinguish in the notation,
we will apply an upper index, and denote by $C^h(S,U)$ the  complexity of the equilibrium class $(S,U)^E$ among $h$-skeletons. (To match earlier notation, in the case of homogeneous polyhedra the upper index 3 will be omitted).

In the case of homogeneous polyhedra, the complexity for all non-monostatic equilibrium classes $(S,U)^E$ for $S,U>1$ has been computed in \cite{balancing}. On the other hand, the complexity has not yet been determined for any of the monostatic classes $(1,U)^E, (S,1)^E$. Lower and upper bounds exist for $C(S,1),C(1,U)$ for $S,U>1$. The most difficult appears to be the mono-monostatic class $(1,1)^E$ for the complexity $C(1,1)$ of which  the prize USD $1,000,000/C(1,1)$ has been offered in \cite{balancing}. Not only is $C(1,1)$  unknown, but also, at this point, there is no known upper bound either.

\subsection{Complexity of some monostable and mono-unstable polyhedral 0-skeletons.}\label{ss:complexity}
Admittedly, computing upper bounds for 0-skeletons is easier.
This is already apparent in the planar case, where monostatic objects with homogeneous mass distribution in the interior do not exist \cite{DomokosRuina}, whereas a monostatic polyhedral 0-skeleton could be constructed with $V=11$ vertices (see Proposition \ref{prop1} using the ideas of \cite{Dawson}). In 3D,  our construction of a 0-skeleton with $F=21$ faces and $V=21$ vertices (see the top left polyhedron in Figure~\ref{fig:compl} and numerical data in Table~\ref{tb:monomono}, line 3) offers such an upper bound as
\begin{equation}\label{bound}
C^0(1,1)\leq 2(21+21-1-1) = 80.
\end{equation}
This is the first known such construction and its existence may help to solve the more difficult cases, in particular,
the case with uniform density. In  Figure~\ref{fig:compl} we provide an illustration for upper bounds for the complexity of polyhedral 0-skeletons in some other monostatic equilibrium classes as well.

\begin{remark}
As we have seen, constructing mono-monostatic 0-skeletons appears to be an easier task as compared to the construction of homogeneous mono-monostatic polyhedra. This can also be expressed by saying that we expect the mechanical complexity of the latter to be much larger than the mechanical complexity of the former. The task of construction of mono-monostatic polyhedra appears to be even easier with \emph{arbitrary mass distribution}. While we do not know the mechanical complexity associated with this class, an upper bound is readily provided by using Conway $(n,k)$-spirals. The latter have at least 7 faces and 7 vertices (as $n\geq 2, k\geq 3, V=F=nk+1$) and the Conway $(2,3)$-spiral, with center of mass at $x_c=y_c=0, z_c<0$ is indeed an example of  mono-monostatic polyhedron. This gives an upper bound of 24 for the mechanical complexity of mono-monostatic polyhedra with arbitrary mass distribution.
\end{remark}

\begin{figure}[ht]
\begin{center}
\includegraphics[width=\textwidth]{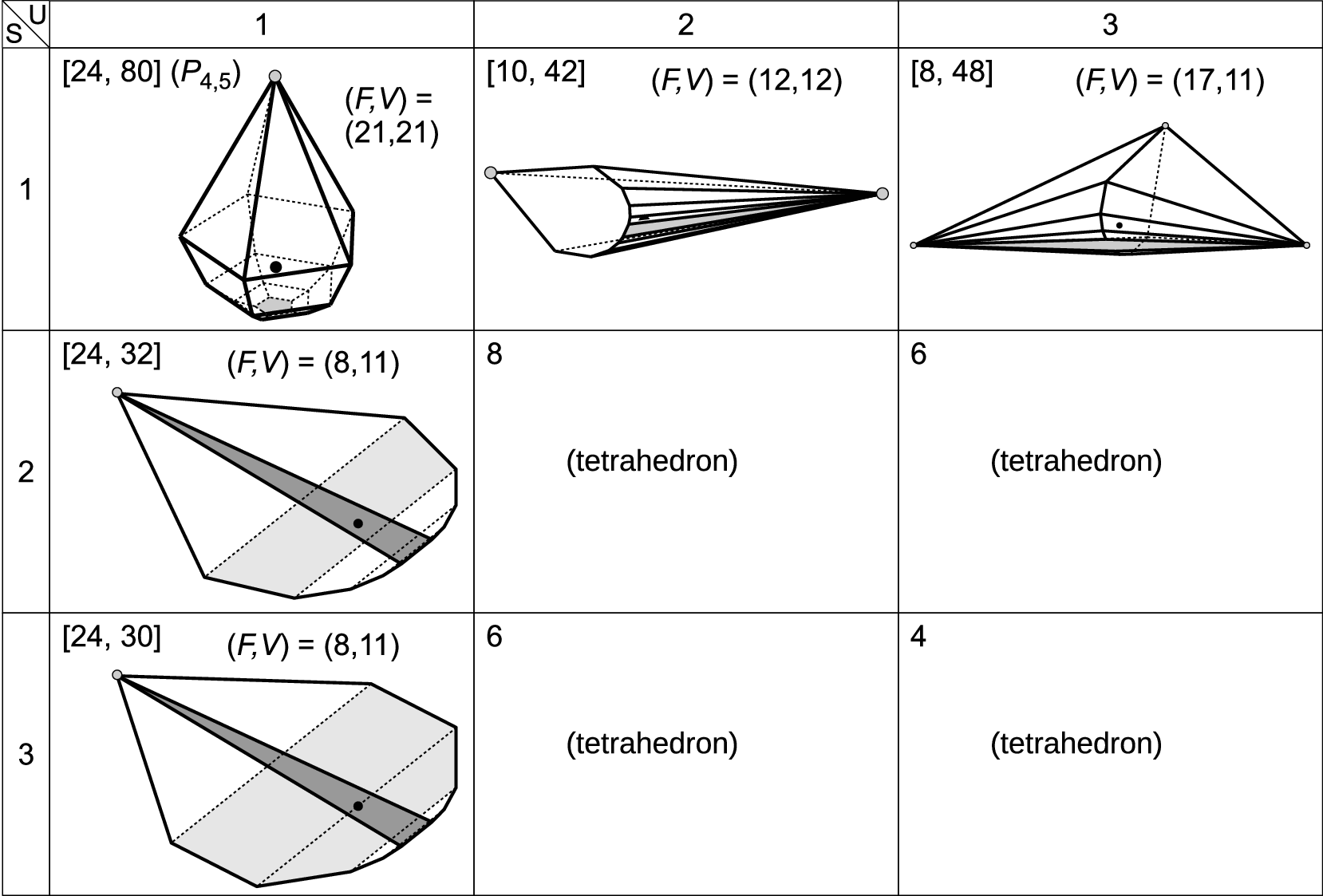}
\caption{Complexity of some monostable and mono-unstable polyhedra. Drawn representatives of equilibrium classes $(S,U)^E$ provide an upper bound for complexity of the respective class, see the bracketed numbers as lower and upper bounds, respectively, in the top left corner of their cells. Since mono-unstable polyhedra with less than 8 vertices (and therefore, by Steinitz's theorem, with less than 6 faces) cannot exist, $26-2S$ is a lower bound of complexity of classes $(S,1)^E$. Complexity of the four non-monostatic classes is exactly known by the existence of simplicial representatives of each class \cite{balancing}. Coordinates of drawn polyhedra, except for the one in class $(1,1)^E$, are given in Table~\ref{tb:coord}.}
\label{fig:compl}
\end{center}
\end{figure}

\subsection{Non-existence of mono-monostatic homogeneous polyhedra generated by Conway $(n,k)$-spirals.}\label{ss:existence}

First we briefly discuss the planar case. In two dimensions we proved (Proposition \ref{prop1}, see also Figure~\ref{fig:intro}b) that double Conway spirals, interpreted as 0-skeletons, can be mono-monostatic. However, it is known from \cite{DomokosRuina} that no homogeneous, convex mono-monostatic objects exist in two dimensions, and this implies that double Conway spirals interpreted as homogeneous objects may never be mono-monostatic. 

Now we show that in three dimensions the situation is similar: we proved (Theorem \ref{th1}) that 0-skeletons generated by Conway $(n,k)$-spirals may be mono-monostatic. Next we show that homogeneous polyhedra generated by  Conway $(n,k)$-spirals  can never be mono-monostatic.

\begin{theorem}
\label{thm:nonexst}
Let $P$ be a convex solid with center of mass at $c$. Let $a$ denote a straight line intersecting $P$ and containing a point $O$, and let $h(a)$ be a half-plane the boundary of which is $a$. Let $N$ denote the intersection of $P$ and $h(a)$ and let us describe $N$ as the polar distance $r(\theta)$, measured from $O$ as origin.

If there exists a straight line $a$ such that $r(\theta)$ is strictly monotonic for all possible $h(a)$ then $c$ cannot coincide with $O$.
\end{theorem}
\begin{proof}
Let an axis $z$ be directed along $a$ and let a point $Q$ on the surface of $P$ be parametrized as $Q(\theta, \varphi)$ where $0 \leq\theta\leq \pi$ is the meridian angle between $OQ$ and $z$, $0\leq\varphi< 2\pi$ is the azimuth angle (with respect to a fixed starting position), and let  $r(\theta, \varphi)=\lvert Q-O \rvert$. 
Since $P$ is convex, $r = r(\theta, \varphi)$ for all surface points is uniquely defined. In this polar system, $c$ can only be the centre of mass of $P$ if
the first static momentum is balanced:
\begin{equation}
\label{eq:Cheight}
    \int\limits_{0}^{2\pi} \int\limits_{0}^{\pi} \dfrac{2}{9} r(\theta, \varphi)^4 \sin\theta \cos\theta d\theta d\varphi = 0,
\end{equation}
where $(1/3)r^3 \sin\theta d\theta d\varphi$ is the volume of an infinitesimal pyramid with its apex at $c$ and $(2/3)r\cos \theta$ measures the $z$ coordinate for the centre of mass of an infinitesimal pyramid.
By the condition of the theorem, $r$ is strictly monotonic in $\theta$. Let us assume that $r$ is strictly monotonically decreasing, i.e., that
$\theta_1 < \theta_2 \iff r_1 > r_2$ for any fixed $\varphi$ and rewrite (\ref{eq:Cheight}) as follows:
\begin{equation*}
    \dfrac{2}{9} \int\limits_{0}^{2\pi} \int\limits_{0}^{\pi/2} \left(r(\theta, \varphi)^4 \sin\theta \cos\theta + r(\pi-\theta, \varphi)^4 \sin(\pi-\theta) \cos(\pi-\theta) \right) d\theta d\varphi = 0
\end{equation*}
\begin{equation}
    \dfrac{1}{9} \int\limits_{0}^{2\pi} \int\limits_{0}^{\pi/2} \left(r(\theta, \varphi)^4 - r(\pi-\theta, \varphi)^4 \right) \sin 2\theta d\theta d\varphi = 0.
\end{equation}
Here both terms of the product in the integrand are positive, so the definite integral cannot evaluate to zero.
\end{proof}

\begin{corollary}
Homogeneous polyhedra generated by Conway $(n,k)$-spirals are never mono-monostatic.
\end{corollary}

\begin{proof}
We prove the Corollary by showing that a Conway $(n,k)$-spiral has either more than two equilibrium points (thus it is not mono-monostatic) or, it satisfies the monotonicity condition of the theorem with $O \equiv c$ (thus it is not mono-monostatic either). Consider $a$ to be aligned with axis $z$. We will refer to planes containing the $z$ axis as \emph{central vertical planes} and intersections of the Conway $(n,k)$-spiral with central vertical planes as \emph{central vertical sections}. The $r=\mbox{\textit{constant}}$ lines are (parts of) concentric circles on all faces. We will treat the single horizontal face ($k$-gon at the bottom) and the $nk$ non-horizontal faces ($k(n-1)$ quadrangles and $k$ triangles) separately. On the horizontal face, the perpendicular projection of $c$ is incident to the $z$-axis, so $r$ increases monotonically along any central vertical section within that face. On each non-horizontal face $f$, there are two, mutually exclusive possibilities: (A) there exists an $r=\mbox{\textit{constant}}$ line which is tangent to  a central vertical plane or (B) there is no such $r=\mbox{\textit{constant}}$ line. If (A) holds then there will be at least one equilibrium point in the interior or on the boundary of $f$. If (B) holds then $r$ will be strictly monotonic on $f$ along any central vertical section.
If (A) is true for \emph{any} non-horizontal face then the Conway $(n,k)$-spiral is not mono-monostatic. If (A) is not true for any of the non-horizontal faces then (B) is true for all of them, so $r$ will be strictly monotonic globally, which, by Theorem {\ref{thm:nonexst}} implies that the Conway $(n,k)$-spiral is not mono-monostatic.
\end{proof}

Theorem~\ref{thm:nonexst} also implies the following
\begin{corollary}
Let $K$ be a homogeneous, smooth convex body with revolution symmetry. Then $K$ cannot be mono-monostatic.
\end{corollary}
\begin{proof}

Let $z$ be the axis of revolution symmetry, let $r(\theta, \varphi)$ be the radial distance function defining the boundary of $K$, measured from the center of mass $c$, and let $0 \leq\theta\leq \pi$ be the meridian angle with respect to the $z$ axis, thus the meridian of $K$ is defined
by $r(\theta)$, $\theta \in [0,\pi]$. Due to symmetry, we have two equilibria at the poles with $\theta=0, \theta =\pi$ for which $r(0) \leq r(\pi)$ can be assumed. 
Let $r(\theta)$ have $n$ stationary points in the interior of $[0,\pi]$. If $n>0$ 
then there are $n$ \emph{rings} of degenerate equilibria, so $K$ is not mono-monostatic.
If $n=0$ then $r(\theta)$ is strictly monotonic and this is the situation described in Theorem \ref{thm:nonexst} and thus $K$ is not mono-monostatic.
So $K$ cannot be mono-monostatic for any value of $n$.
\end{proof}

\section{Concluding comments.}\label{sum}
In this paper, by relying on the geometric idea of Conway spirals,  we demonstrated the existence of mono-monostatic 0-skeletons in two and three dimensions. In the former case, by drawing on an earlier result of Dawson \cite{Dawson} we showed
that mono-monostatic planar 0-skeletons with $V=11$ vertices exist. It follows from another result of Dawson \cite{Dawson3} that for $V=9$ such constructions cannot exist. The $V=10$ case is not known.  In three dimensions we showed an explicit construction with $V=21$ vertices,
thus providing an upper bound for the minimal number of vertices. The lower bound is $V=8$ \cite{Bozoki} and other results are not known.
We hope that these constructions will motivate further research to find the minimal number of $V$ for a mono-monostatic 0-skeleton, both in two and three dimensions.

Conway $(n,k)$-spirals played a central role in this article. It might be of interest to note that for odd values of $k$,
such a polyhedron is combinatorially equivalent to a \emph{strongly self-dual polyhedron} \cite{Ghorvath}, \cite{Lovasz}.

\begin{table}[ht]
\begin{center}
\begin{scriptsize}
    \begin{tabular}{|r|r|r|}
    \hline
    \multicolumn{3}{|c|}{$(1,2)^S$}\\
    $x$ & $y$ & $z$ \\
    \hline
     0 & 374 & 0 \\
    154 & 80 & 0 \\
    124 & -32 & 0 \\
    81 & -78 & 0 \\
    47 & -95 & 0 \\
    24 & -100 & 0 \\
    -24 & -100 & 0 \\
    -47 & -95 & 0 \\
    -81 & -78 & 0 \\
    -124 & -32 & 0 \\
    -154 & 80 & 0 \\
    0 & -1200 & 5000 \\
    \hline
    \end{tabular}    
    \begin{tabular}{|r|r|r|}
    \hline
    \multicolumn{3}{|c|}{$(1,3)^S$}\\
    $x$ & $y$ & $z$ \\
    \hline
     0 & 466 & 0 \\
    166 & 70 & 0 \\
    121 & -47 & 0 \\
    71 & -87 & 0 \\
    35 & -100 & 0 \\
    -35 & -100 & 0 \\
    -71 & -87 & 0 \\
    -121 & -47 & 0 \\
    -166 & 70 & 0 \\
    0 & -100 & -900 \\
    0 & -100 & 900 \\
    \hline
    \multicolumn{3}{c}{\phantom{0} }
\end{tabular}
\vskip 5mm
    \begin{tabular}{|r|r|r|}
    \hline
    \multicolumn{3}{|c|}{$(2,1)^S$}\\
    $x$ & $y$ & $z$ \\
    \hline
    0 & 374.328 & 0 \\
    153.589 & 80.2023 & 20 \\
    124.268 & -32.3675 & 14.9819 \\
    81.1006 & -77.5258 & 8.45141 \\
    46.9121 & -94.4981 & 3.41302 \\
    23.4562 & -100 & 0 \\
    -23.4562 & -100 & 0 \\
    -46.9121 & -94.4981 & 3.41302 \\
    -81.1006 & -77.5258 & 8.45141 \\
    -124.268 & -32.3675 & 14.9819 \\
    -153.589 & 80.2023 & 20 \\
    \hline
    \multicolumn{3}{c}{\phantom{0} }
    \end{tabular}    
    \begin{tabular}{|r|r|r|}
    \hline
    \multicolumn{3}{|c|}{$(3,1)^S$}\\
    $x$ & $y$ & $z$ \\
    \hline
    0 & 334.907 & 0 \\
    145.019 & 83.7267 & 10 \\
    145.019 & 0 & 9.6018 \\
    94.9161 & -68.9606 & 5.40618 \\
    53.5898 & -92.8203 & 2.10256 \\
    26.7949 & -100 & 0 \\
    -26.7949 & -100 & 0 \\
    -53.5898 & -92.8203 & 2.10256 \\
    -94.9161 & -68.9606 & 5.40618 \\
    -145.019 & 0 & 9.6018 \\
    -145.019 & 83.7267 & 10 \\
    \hline
    \multicolumn{3}{c}{\phantom{0} }
    \end{tabular}    
\end{scriptsize}
\caption{Coordinates of some polyhedra shown in Figure~\ref{fig:compl}. Monostable objects are provided with integer coordinates which would be difficult for mono-unstable ones due to oblique polygonal faces.}
\label{tb:coord}
\end{center}
\end{table}

\vfill
\end{document}